\documentclass[10pt]{amsart}
\usepackage{cases,mathtools,accents,verbatim}

\theoremstyle{plain}
\newtheorem{theorem}{Theorem}
\newtheorem{proposition}{Proposition}
\newtheorem{corollary}{Corollary}
\newtheorem{lemma}{Lemma}

\theoremstyle{definition}
\newtheorem{definition}{Definition}
\newtheorem{example}{Example}
\newtheorem{remark}{Remark}

\newcommand{\rn}{\mathbb R}
\newcommand{\nn}{\mathbb N}
\newcommand{\qn}{\mathbb Q}
\newcommand{\zn}{\mathbb Z}
\newcommand{\cn}{\mathbb C}
\newcommand{\sn}{\mathbb S}

\newcommand{\la}{\langle}
\newcommand{\ra}{\rangle}

\renewcommand{\i}{\mathrm{i}}

\DeclareMathOperator{\ricci}{Ricci}

\DeclareMathOperator{\tr}{trace}

\DeclareMathOperator{\sgn}{sgn}
\DeclareMathOperator{\rank}{rank}

\begin{document}

\title[CMC biharmonic surfaces]{CMC proper-biharmonic surfaces of constant Gaussian curvature in spheres}

\author{E. Loubeau}
\address{D{\'e}partement de Math{\'e}matiques \\
LMBA, UMR 6205 \\
Universit{\'e} de Bretagne Occidentale \\
6, avenue Victor Le Gorgeu \\
CS 93837, 29238 Brest Cedex 3, France}
\email{Eric.Loubeau@univ-brest.fr}

\author{C. Oniciuc}
\address{Faculty of Mathematics, Al.I. Cuza University of Iasi, Bd. Carol I no. 11, 700506
Iasi, Romania}
\email{oniciucc@uaic.ro}

\keywords{Biharmonic map; Constant mean curvature surfaces}
\subjclass{53C42, 53C43, 58E20}
\thanks{C. Oniciuc was supported by a grant of the Romanian National Authority for Scientific Research, CNCS-UEFISCDI, project number PN-II-RU-TE-2011-3-0108}

\begin{abstract}
CMC surfaces in spheres are investigated under the extra condition of biharmonicity. From the work of Miyata, especially in the flat case, we give a complete description of such immersions and show that for any $h\in (0,1)$ there exist CMC proper-biharmonic planes and cylinders in $\sn^5$ with $|H|=h$, while a necessary and sufficient condition on $h$ is found for the existence of CMC proper-biharmonic tori in $\sn^5$.
\end{abstract}

\maketitle

\section{Introduction}

While the link between minimal submanifolds and harmonic Riemannian immersions is by now a classical fact, the generalisation of harmonicity by biharmonic maps as critical points of the bienergy
$$ E_2 (\phi) = \tfrac{1}{2} \int_M |\tau(\phi)|^2 \, v_g,$$
for $\phi : (M,g) \to (N,\tilde{g})$, has many contact points with CMC submanifolds but their relationship is a more complex one. 

The Euler-Lagrange equation for biharmonic maps
$$\tau_2 (\phi) = - \Delta \tau(\phi) - \tr R^N (d\phi,\tau(\phi))d\phi =0,$$
certainly simplifies when considering Riemannian immersions, but not to the point of revealing any specific geometric content, and therefore it remains worthwhile to combine this variational problem with more classical conditions. By proper-biharmonic, we shall designate maps which are biharmonic without being harmonic.

Submanifolds in spheres which, as submanifolds in the ambient Euclidean space, admit a decomposition into a sum of one or two eigenfunctions of the Laplacian, according to their type, were intensively studied.

By a result of Takahashi~\cite{Takahashi}, it is well-known that a submanifold of the Euclidean sphere $\sn^n$ is of $1$-type if and only if it is minimal in $\sn^n$ or in a small hypersphere of $\sn^n$. Further, the mass-symmetric $2$-type surfaces, which must be CMC, were investigated for example in~\cite{HV1,HV2,Miyata}.

The condition of biharmonicity on CMC submanifolds in spheres implies this remarkable decomposition of the immersion. We will always consider full immersions in $\sn^n$, i.e. their images do not lie in any totally geodesic sphere $\sn^{n'}\subset \sn^n$. From the work of Miyata~\cite{Miyata}, the various parameters involved in this description of mass-symmetric $2$-type flat surfaces in $\sn^5$ simplify in the biharmonic case into formulas given in Theorem~\ref{thm6} and the immersions merely depend on an angle $\rho\in [0,\tfrac{1}{2}\arccos\tfrac{h-1}{1+h}]$. 

Such formulas enable us to conclude that for any $h\in (0,1)$, there exist CMC proper-biharmonic planes and cylinders, of constant mean curvature $h$, in the unit Euclidean $5$-dimensional sphere $\sn^5$.

Moreover, double periodicity and therefore the existence of CMC proper-biharmonic tori in $\sn^5$, is dependent on $h$ being in the range of a certain rational function on $\qn^2$.

CMC proper-biharmonic surfaces have also been studied in~\cite{Fetcu,OZ,S}.

For Riemannian immersions in spheres, the condition of biharmonicity combined with CMC forces the mean curvature to be less than one, with the extreme values zero and one being characteristic of minimality, a trivial case for our study, and, for the latter, a well-known construction of biharmonic maps from minimal immersions in a $45$-th parallel hypersphere.

\begin{proposition}\label{prop1}\cite{OniciucPhD}
Let $\phi : M^m \to \sn^n$ be a CMC proper-biharmonic immersion. Then $|H|\in (0,1]$ and $|H|=1$ if and only if $\phi$ induces a minimal immersion of $M^m$ into a small hypersphere $\sn^{n-1}(\frac{1}{\sqrt{2}}) \subset \sn^n$.
\end{proposition}

The next proposition illustrates how too strong a condition on the mean curvature brings us back to a familiar case.

\begin{proposition}\cite{BMO2}
Let $\phi : M^2 \to \sn^n$ be a parallel mean curvature proper-biharmonic surface. Then $\phi$ induces a minimal immersion of $M^2$ into a small hypersphere $\sn^{n-1}(\frac{1}{\sqrt{2}}) \subset \sn^n$.
\end{proposition}

The usual decomposition of smooth functions over eigenfunctions of the Laplacian, admits a counterpart for Riemannian immersions, particularly relevant on compact manifolds.

\begin{definition}\cite{Chen1,Chen2}
A Riemannian immersion $\psi : (M^m,g) \to \rn^{n+1}$ is called of finite type if it can be expressed as a finite sum of $\rn^{n+1}$-valued eigenmaps of the Laplacian $\Delta$ of $(M,g)$, i.e.
\begin{equation}\label{specdec}
\psi = \psi_0 + \psi_{t_1} + \cdots +\psi_{t_k} ,
\end{equation}
where $\psi_0 \in \rn^{n+1}$ is a constant vector and $\psi_{t_i} : M \to \rn^{n+1}$ are non-constant maps satisfying $\Delta \psi_{t_i} =\lambda_{t_i}\psi_{t_i}$, for $i = 1,\dots,k$. If, in particular, all eigenvalues $\lambda_{t_i}$ are mutually distinct, the submanifold is said to be of $k$-type and \eqref{specdec} is called the spectral decomposition of $\psi$.
\end{definition}

If $M$ is compact, the Riemannian immersion $\psi : (M^m,g) \to \rn^{n+1}$ admits a unique spectral decomposition 
$$\psi = \psi_0 + \sum_{i=1}^{+\infty} \psi_{i}, \quad \Delta \psi_{i} =\lambda_{i}\psi_{i} ,\quad \lambda_i >0 ,$$
where $\psi_0$ is called the centre of mass. Then, it is of $k$-type if and only if $k$ terms of $\{\psi_i \}_{i=1}^{+\infty}$ do not vanish. The centre of mass is the harmonic component of the spectral decomposition.

If $M$ is not compact, the spectral decomposition 
$$\psi = \psi_0 + \sum_{i=1}^{+\infty} \psi_{i}, \quad \Delta \psi_{i} =\lambda_{i}\psi_{i} ,$$
is not guaranteed and the harmonic component is not necessarily constant. However, we may agree that, if we have a decomposition~\eqref{specdec} and there exists $t_{i_0}$ such that 
$\lambda_{t_{i_0}} =0$, then we redenote $\psi_0 + \psi_{t_{i_0}} = \psi_{t_{i_0}}$. With this convention in mind, two decompositions~\eqref{specdec} must agree.

The starting point of this work is the correspondence between proper-biharmonic immersions in spheres and maps of finite type, one or two according to the value of the mean curvature, which by Proposition~\ref{prop1} must be less than one.

\begin{theorem}\cite{BMO1,BMO2}
Let $\phi : (M^m,g) \to \sn^{n}$ be a proper-biharmonic immersion. Denote by $\psi = i\circ\phi : M \to \rn^{n+1}$ the immersion of $M$ in $\rn^{n+1}$, where $i : \sn^{n}\to \rn^{n+1 }$ is the canonical inclusion map. Then
\begin{enumerate}
\item[i)] The map $\psi$ is of $1$-type if and only if $|H|=1$. In this case, $\psi = \psi_0 + \psi_{t_1}$, with $\Delta \psi_{t_1} = 2m \psi_{t_1}$, $\psi_{0}$ is a constant vector. Moreover, $\langle \psi_0 , \psi_{t_1} \rangle =0$ at any point, $|\psi_{0}| =|\psi_{t_1}| = \tfrac{1}{\sqrt{2}}$ and
$\phi_{t_1}:M\to \sn^{n-1}\left(\frac{1}{\sqrt{2}}\right)$ is a minimal immersion.
\item[ii)] The map $\psi$ is of $2$-type if and only if $|H|$ is constant, $|H|\in (0,1)$. In this case $\psi = \psi_{t_1} + \psi_{t_2}$, with $\Delta \psi_{t_1} = m(1-|H|)\psi_{t_1}$,  
$\Delta \psi_{t_2} = m(1+|H|)\psi_{t_2}$ and 
\begin{align*}
\psi_{t_1} &= \tfrac{1}{2} \psi + \tfrac{1}{2|H|} H \, , \,  \psi_{t_2} = \tfrac{1}{2} \psi - \tfrac{1}{2|H|} H .
\end{align*}
Moreover, $\langle \psi_{t_1} , \psi_{t_2} \rangle =0$, $|\psi_{t_1}| =|\psi_{t_2}| = \tfrac{1}{\sqrt{2}}$ and
$$ \phi_{t_i} : (M,g) \to \sn^{n}\left(\tfrac{1}{\sqrt{2}}\right), \quad i=1,2,$$ 
are harmonic maps with constant density energy.
\end{enumerate}
\end{theorem}

Harmonic maps are characterised as solutions to (a system of) second-order elliptic equations of the same symbol as the Laplacian on functions, and a unique continuation property for harmonic maps is to be expected. This was proved by J.~H.~Sampson in 1978.

\begin{theorem} \cite{Sampson}
Let $\phi_1, \phi_2 : (M^m ,g) \to (N^n ,\tilde{g})$ be two harmonic maps. If $\phi_1$ and $\phi_2$ agree on an open subset of $M$, then they agree everywhere. In particular a harmonic map constant on an open subset is a constant map.
\end{theorem}

The last part of the theorem was generalised to biharmonic maps.

\begin{theorem}\cite{CMO}
Let $\phi : (M^m ,g) \to (N^n ,\tilde{g})$ be a biharmonic map. If $\phi$ is harmonic on an open subset then it is a harmonic map on the whole of $M$.
\end{theorem}

A deep difference between harmonic and biharmonic maps is that the latter are solutions to (a system of) fourth-order elliptic equations for which no maximum principle exists. Therefore, a biharmonic analogue of Sampson's Theorem is out of reach but, if combined with CMC, thanks to their type decomposition, we can prove a unique continuation for Riemannian immersions into spheres.

\begin{theorem}
Let $\phi_1, \phi_2 : (M^m ,g) \to \sn^n$ be two CMC proper-biharmonic immersions. If $\phi_1$ and $\phi_2$ agree on an open subset of $M$, then they agree everywhere.
\end{theorem}

\begin{proof}
Let $U$ be an open subset of $M$ such that $\phi_1$ and $\phi_2$ agree on $U$. Then the mean curvature vector fields $H_1$ and $H_2$ agree on $U$ and therefore $|H_1|=|H_2|$ on M. For the case $|H_1|=|H_2|=1$, we denote
$$\psi_1 = \psi_{1,0} + \psi_{1,t_1} , \quad \psi_2 = \psi_{2,0} + \psi_{2,t_1} ,$$
where $\psi_i = i\circ \phi_i : M \to \rn^{n+1}$, $i=1,2$. As $\psi_1$ and $\psi_2$ agree on $U$, we deduce that $\psi_{1,0}= \psi_{2,0}$ and $\psi_{1,t_1}= \psi_{2,t_1}$ on $U$. Since $\Delta \psi_{1,t_1}=2m\psi_{1,t_1}$ and $\Delta \psi_{2,t_1}=2m\psi_{2,t_1}$, therefore $\psi_{1,t_1}= \psi_{2,t_1}$ on $M$ and $\phi_1=\phi_2$ on the whole of $M$. The case $|H_1|=|H_2|\in (0,1)$ is analogous.
\end{proof}

The next result on CMC proper-biharmonic immersions follows directly from a well-known property of harmonic maps.

\begin{proposition}
Let $\phi : (M^m ,g) \to \sn^n$ be a CMC proper-biharmonic immersion. If an open subset of $M$ is mapped into a totally geodesic $\sn^{n'}$, $0<n' <n$, then $\phi(M)\subset \sn^{n'}$.
\end{proposition}

Minimal immersions into spheres can be summed up to produce new maps into higher dimensional spheres and they will turn out to describe thoroughly one case of CMC proper-biharmonic immersions, namely those with positive constant Gaussian curvature (see Proposition~\ref{prop8}).

Let $\phi_1 : (M^m , g) \to \sn^{n_1}(r_1)$ and $\phi_2 : (M^m , g) \to \sn^{n_2}(r_2)$ be two minimal immersions, where $r_1$ and $r_2$ are positive constants. Consider the map
\begin{align*}
\phi &: M^m \to \sn^{n_1 + n_2 +1} \, , \, \phi(p) = (\alpha \phi_{1}(p) , \beta \phi_{2}(p)) ,
\end{align*}
where $\alpha$ and $\beta$ are real numbers such that $\alpha^2 r_1^2 + \beta^2 r_2^2 =1$.

The map $\phi$ is a Riemannian immersion if and only if $\alpha^2 + \beta^2 =1$.

Assume that
\begin{equation*}
\begin{cases}
\alpha^2 r_1^2 + \beta^2 r_2^2 =1 \\
\alpha^2 + \beta^2 =1
\end{cases}
\end{equation*}

It is not difficult to see that $\phi$ is also pseudo-umbilical, i.e.
$$ A_{H} = |H|^2 I = \left( \frac{\alpha^2}{r_1^2} + \frac{\beta^2}{r_2^2} -1\right) I .$$

Straightforward computations yield easy conditions on the parameters so that the resulting diagonal sum be indeed a Riemannian immersion into a sphere.

\begin{proposition} \label{prop2}
The Riemannian immersion
\begin{align*}
\phi &: M^m \to \sn^{n_1 + n_2 +1} \, , \, \phi(p) = (\alpha \phi_{1}(p) , \beta \phi_{2}(p)) ,
\end{align*}
is proper-biharmonic if and only if 
\begin{equation*}
\begin{cases}
\alpha^2 = \frac{1}{2r_1^2} , \quad  \beta^2 =\frac{1}{2r_2^2}, \\
\frac{1}{r_1^2} + \frac{1}{r_2^2} = 2, \quad r_1\neq r_2.
\end{cases}
\end{equation*}
In this case $|H|^2=1 - \frac{1}{r^2_1r^2_2}$.
\end{proposition}

\begin{proof}
By a straightforward computation we get:
\begin{align*}
\tau(\phi) &= m\left( \alpha\left( 1 - \frac{1}{r_1^2}\right) \phi_1 , \beta\left( 1 - \frac{1}{r_2^2}\right) \phi_2\right) ,\\
\tau_{2}(\phi) &= 
m^2 \left( \alpha^2 \left( 1 - \frac{1}{r_1^2}\right) + \beta^2 \left( 1 - \frac{1}{r_2^2}\right)\right) (\alpha \phi_{1} , \beta \phi_{2} ) \\
& + m^2\left( \alpha\left( 1 - \frac{1}{r_1^2}\right)^2 \phi_1 , \beta\left( 1 - \frac{1}{r_2^2}\right)^2 \phi_2\right) .
\end{align*}
\end{proof}

\section{CMC proper-biharmonic surfaces in $\sn^n$}

The case of CMC proper-biharmonic surfaces with mean curvature $h\in (0,1)$ in unit Euclidean spheres $\sn^n$ is, as expected, better understood than the general case of an arbitrary dimension, and several rigidity phenomenons can be observed. 

We first show in Proposition~\ref{prop5} that if such a surface has constant Gaussian curvature, it must be non-negative and, in case of compactness, less than one and previous work, based on the stress-energy tensor of biharmonic maps, has led, for spheres and planes, to even more stringent restrictions, cf. Propositions~\ref{prop6} and \ref{prop7}. 

Moreover, results from Miyata~\cite{Miyata} can be re-examined under the condition of biharmonicity, to yield very precise descriptions of CMC proper-biharmonic immersions from a surface of positive Gaussian curvature (Proposition~\ref{prop8}) or flat (Theorem~\ref{thmMiyata}), while negative Gaussian curvature is ruled out by a result of Bryant~\cite{Bryant,Miyata}.

Since diagonal sums of Boruvka spheres give all the CMC proper-biharmonic surfaces of positive Gaussian curvature in spheres, the really interesting case lies in flat surfaces.

When the target has dimension five, they are entirely described by parameters whose inter-dependence is given in Lemma~\ref{lemma1}, Theorem~\ref{thm6} and Lemma~\ref{lemma2}.

An immediate consequence of these results is that for any $h \in (0,1)$, there exist CMC (equal to $h$) proper-biharmonic immersions of a plane (Corollary~\ref{cor2}) and, because of the continuous range of the parameters, infinitely many will admit a periodicity and therefore quotient to a cylinder in $\sn^5$ (Corollary~\ref{cor3}).

Furthermore, a necessary and sufficient condition on the value of $h$ is found in Theorem~\ref{thm7}, to the existence of a proper-biharmonic torus of CMC $h$ in $\sn^5$.

The first result is based on a classical inequality on the eigenvalues of the Laplacian which constrains the geometry of CMC proper-biharmonic immersions.

\begin{proposition}\label{prop5}
Let $\phi : M^2 \to \sn^n$ be a CMC proper-biharmonic immersion. Assume $M^2$ has constant Gaussian curvature $K^M$ and $|H|\in(0,1)$. Then $K^M \geq 0$. Moreover, if $M$ is compact then $K^M\in [0,1)$.
\end{proposition}

\begin{proof}
Since $|H|\in(0,1)$, $M$ is of type $2$ and, according to a result of Miyata~\cite[Theorem B]{Miyata}, $K^M \geq 0$. Assume now that $M$ is compact. Since $\psi = i\circ \phi = \psi_{t_1} + \psi_{t_2}$, $\Delta \psi_{t_1} = 2(1-|H|)\psi_{t_1}$, we have $\lambda_1 \leq 2(1-|H|) < 2$. But, by a result of Lichnerowicz, as $\ricci^M = K^M g$, $\lambda_1 \geq 2 K^M$. Thus $K^M<1$.
\end{proof}

Using the special properties of the stress-energy tensor of biharmonic maps on CMC surfaces, we can show parallelism of the shape operator when the Gaussian curvature is non-negative.

\begin{proposition}\label{prop6} \cite{LO}
Let $\phi : \sn^{2}(r) \to \sn^n$ be a CMC proper-biharmonic immersion. Then $\sn^{2}(r)$ is pseudo-umbilical in $\sn^n$.
\end{proposition}

\begin{proposition}\label{prop7} \cite{LO}
Let $\phi : \rn^{2} \to \sn^n$ be a CMC proper-biharmonic immersion. Then $\nabla A_H = 0$.
\end{proposition}

One effect of Miyata's work on surfaces of positive Gaussian curvature is a description in terms of diagonal sums of Boruvka spheres.

\begin{proposition}\label{prop8}
A Riemannian immersion $\phi$ is a CMC proper-biharmonic map from a surface with positive Gaussian curvature in $\sn^{n}$ if and only if it is (the restriction to an open subset of) the diagonal sum  $\phi= (\alpha \phi_1 , \beta \phi_2)$, where 
\begin{align*}
\phi_1 &: \sn^{2}(r) \to \sn^{2n_1}(r_1) \, , \, \phi_2 : \sn^{2}(r) \to \sn^{2n_2}(r_2)
\end{align*}
are Boruvka minimal immersions with parameters
\begin{align*}
\alpha^2 &=  \frac{q_1}{q_1 + q_2} \quad \mbox{and} \quad \beta^2 = \frac{q_2}{q_1 + q_2} \\
r_1 &= \sqrt{\frac{q_1 + q_2}{2q_1}} \quad \mbox{and} \quad r_2 = \sqrt{\frac{q_1 + q_2}{2q_2}} \\
r &= \frac{1}{2} \sqrt{q_1 + q_2} ,
\end{align*}
with $q_1 =n_1 (n_1 +1)$ and $q_2 =n_2 (n_2 +1)$, and $n_1\neq n_2$. Moreover,
$|H|^2=\frac{(q_1 - q_2)^2}{(q_1 + q_2)^2}$ and $\phi$ is pseudo-umbilical.
\end{proposition}

\begin{proof}
By \cite[Theorem A]{Miyata}, the only mass-symmetric $2$-type immersions of positive Gaussian curvature in $\sn^{n}$ are diagonal sums of two different standard minimal immersions of two-spheres.

Let $\phi_1 : \sn^{2}(r) \to \sn^{2n_1}(r_1)$, $r_1 = r \sqrt{\tfrac{2}{n_1 (n_1 +1)}}$, $n_1 \geq 2$, and $\phi_2 : \sn^{2}(r) \to \sn^{2n_2}(r_2)$, 
$r_2 = r \sqrt{\tfrac{2}{n_2 (n_2 +1)}}$, $n_2 \geq 2$, be two Boruvka minimal immersions and consider the diagonal map
$$\phi = (\alpha \phi_1, \beta \phi_2) : \sn^{2}(r) \to \sn^{2n_1+ 2n_2 +1},$$
under the conditions
\begin{equation}
\begin{cases}
\alpha^2 + \beta^2 = 1 ,\\
\alpha^2 r_1^2+ \beta^2 r_2^2 = 1 ,
\end{cases}
\end{equation}
so that $\phi$ is a Riemannian immersion. The proof then follows from Proposition~\ref{prop2}.
\end{proof}

\begin{remark}[case of $K^M<0$]
There can be no CMC proper-biharmonic surface in a sphere with $K^M<0$, by Bryant, cf.~\cite[Theorem B]{Miyata} and \cite{Bryant}.
\end{remark}

When the surface is actually flat, Miyata's formula provides explicit description, once extended to $\rn^2$, in terms of four parameters.
 
\begin{theorem}\label{thmMiyata}\cite{Miyata}
Let $D$ be a small disk about the origin in the Euclidean plane $\rn^2$ and $\phi : D \to \sn^{n}$ be a CMC proper-biharmonic immersion with $|H|\in (0,1)$. Then
\begin{enumerate}
\item[i)] $n$ is odd, $n\geq 5$.
\item[ii)] $\phi$ extends uniquely into a CMC proper-biharmonic immersion of $\rn^2$ into $\sn^{n}$.
\item[iii)] $\psi = i \circ \phi : \rn^2 \to \rn^{n+1}$ can be written
\begin{align}
\psi (z) &= \tfrac{1}{\sqrt{2}} \sum_{k=1}^{m} \sqrt{R_k}\Big(  e^{\tfrac{\sqrt{\lambda_1}}{2}(\mu_k z - \bar{\mu}_k\bar{z})}Z_k
+ e^{\tfrac{\sqrt{\lambda_1}}{2}(-\mu_k z + \bar{\mu}_k\bar{z})}\bar{Z}_k\Big) \notag\\
&+ \tfrac{1}{\sqrt{2}} \sum_{j=1}^{m'} \sqrt{R'_j}\Big( e^{\tfrac{\sqrt{\lambda_2}}{2}(\eta_j z - \bar{\eta}_j \bar{z})}W_j
+ e^{\tfrac{\sqrt{\lambda_2}}{2}(-\eta_j z + \bar{\eta}_j \bar{z})}\overline{W}_j\Big) , \tag{*}\label{formula*}
\end{align}
where 
\begin{enumerate}
\item $Z_k = \tfrac{1}{2} \Big( E_{2k-1} - \i E_{2k}\Big),\, k=1,\dots,m, \i^2=-1$,
\item $W_j = \tfrac{1}{2} \Big( E_{2(m+j)-1} - \i E_{2(m+j)}\Big) ,\, j=1,\dots,m',$
\item $\{E_1,\dots,E_{2m+2m'}\}$ is an orthonormal basis of $\rn^{n+1}$, $n=2m+2m'-1$,
\item $\lambda_1 = 2(1-|H|)$, $\lambda_2 = 2(1+|H|)$, $|H|$ constant, $|H|\in (0,1)$ 
\item $\sum_k R_k =1$, $\sum_j R'_j =1$, $R_k >0$, $R'_j >0$,
\item $(1-|H|)\sum_k \mu^2_k R_k + (1+|H|)\sum_j \eta^2_jR'_j = 0$.
\item $\{ \pm \mu_k\}_{k=1}^{m}$ are $2m$ distinct complex numbers of norm $1$, 
\item $\{ \pm \eta_j\}_{j=1}^{m'}$ are $2m'$ distinct complex numbers of norm $1$.
\end{enumerate}
\end{enumerate}
\end{theorem}

\begin{remark}[symmetries of solutions]
Let $h=|H|$ and $(h,R_k,R'_j,\mu_k,\eta_j)$ be a solution of (e), (f), (g) and (h) then
\begin{itemize}
\item $(h,R_k,R'_j,\pm\mu_k,\pm\eta_j)$ is also a solution of Conditions~(e), (f), (g) and (h) ($2^{m+m'}$ solutions).
\item $(h,R_k,R'_j,\bar{\mu}_k,\bar{\eta}_j)$ is also a solution of Condition~(e), (f), (g) and (h) (1 solution). Everything is up to an isometry of $\sn^n$ and of $\rn^2$.
\item $(h,R_k,R'_j,\alpha \mu_k,\alpha\eta_j),\alpha\in\cn, |\alpha|=1$ is also a solution of Condition~(e), (f), (g) and (h).
\end{itemize}
\end{remark}

While a restricted case, pseudo-umbilical maps on $\rn^2$ are, as for positive Gaussian curvature, completely described by diagonal sums of minimal immersions.
\begin{proposition}
Let $\phi :\rn^2 \to \sn^n$ be a CMC proper-biharmonic immersion with mean curvature $h\in (0,1)$. 
Then it is pseudo-umbilical if and only if $\psi=i\circ\phi$ is a diagonal map and $\psi=\psi_{t_1} + \psi_{t_2}=(\psi_{t_1},\psi_{t_2})$, where 
$\phi_{t_1}:\rn^2\to \sn^{2m-1}\left(\tfrac{1}{\sqrt{2}}\right)$, $\phi_{t_2}:\rn^2\to \sn^{2m'-1}\left(\tfrac{1}{\sqrt{2}}\right)$ are harmonic maps, 
$2(m+m')= n+1$, and $\phi_{t_1}^* \la , \ra = \tfrac{1-h}{2}\la ,\ra$, $\phi_{t_2}^* \la , \ra = \tfrac{1+h}{2}\la ,\ra$.
\end{proposition}

\begin{proof}
By direct computation, we obtain:
\begin{align*}
&(\nabla d\phi)(\partial x,\partial x)= \frac{\partial^2 \psi}{\partial x^2} + \psi\\
&=\tfrac{1}{\sqrt{2}} \sum_{k}  \sqrt{R_k} \Big(e^{\tfrac{\sqrt{\lambda_1}}{2}(\mu_k z - \bar{\mu}_k\bar{z})}
 \Big(\tfrac{\lambda_1}{4}(\mu_k -\bar{\mu}_k )^2 +1\Big) Z_k + e^{\tfrac{\sqrt{\lambda_1}}{2}(-\mu_k z + \bar{\mu}_k\bar{z})}\Big(\tfrac{\lambda_1}{4}(-\mu_k + \bar{\mu}_k )^2 +1\Big)\bar{Z}_k\Big) \\
&+ \tfrac{1}{\sqrt{2}} \sum_{j}  \sqrt{R'_j} \Big(e^{\tfrac{\sqrt{\lambda_2}}{2}(\eta_j z - \bar{\eta}_j \bar{z})}
\Big(\tfrac{\lambda_2}{4}(\eta_j - \bar{\eta}_j )^2 +1\Big) W_j + e^{\tfrac{\sqrt{\lambda_2}}{2}(-\eta_j z + \bar{\eta}_j \bar{z})} \Big(\tfrac{\lambda_2}{4}(-\eta_j + \bar{\eta}_j)^2 +1\Big) \overline{W}_j\Big) ,
\end{align*}
and
\begin{align*}
&(\nabla d\phi)(\partial y,\partial y)=\frac{\partial^2 \psi}{\partial y^2} +\psi  \\
&=\tfrac{1}{\sqrt{2}} \sum_{k} \sqrt{R_k} \Big(  e^{\tfrac{\sqrt{\lambda_1}}{2}(\mu_k z - \bar{\mu}_k\bar{z})}
 \Big(\tfrac{-\lambda_1}{4}(\mu_k +\bar{\mu}_k )^2+1\Big) Z_k + e^{\tfrac{\sqrt{\lambda_1}}{2}(-\mu_k z + \bar{\mu}_k\bar{z})}
\Big(\tfrac{-\lambda_1}{4}(\mu_k + \bar{\mu}_k )^2+1\Big) \bar{Z}_k\Big) \\
&+ \tfrac{1}{\sqrt{2}} \sum_{j} \sqrt{R'_j}\Big( e^{\tfrac{\sqrt{\lambda_2}}{2}(\eta_j z - \bar{\eta}_j \bar{z})}
\Big(\tfrac{-\lambda_2}{4}(\eta_j + \bar{\eta}_j )^2 +1\Big) W_j 
+ e^{\tfrac{\sqrt{\lambda_2}}{2}(-\eta_j z + \bar{\eta}_j \bar{z})} \Big(\tfrac{-\lambda_2}{4}(\eta_j + \bar{\eta}_j)^2+1\Big) \overline{W}_j\Big) .
\end{align*}
So
\begin{align*}
2H &= 
\tfrac{1}{\sqrt{2}} \sum_{k}  \sqrt{R_k}\Big( e^{\tfrac{\sqrt{\lambda_1}}{2}(\mu_k z - \bar{\mu}_k\bar{z})}
 (2-\lambda_1) Z_k +  e^{\tfrac{\sqrt{\lambda_1}}{2}(-\mu_k z + \bar{\mu}_k\bar{z})}(2-\lambda_1)\bar{Z}_k\Big) \\
&+ \tfrac{1}{\sqrt{2}} \sum_{j} \sqrt{R'_j} \Big( e^{\tfrac{\sqrt{\lambda_2}}{2}(\eta_j z - \bar{\eta}_j \bar{z})}
(2-\lambda_2) W_j 
+  e^{\tfrac{\sqrt{\lambda_2}}{2}(-\eta_j z + \bar{\eta}_j \bar{z})} (2-\lambda_2) \overline{W}_j\Big)\\
&= (2-\lambda_1) \psi_1 + (2-\lambda_2)\psi_2\\
&= 2|H| \psi_1 - 2|H| \psi_2 .
\end{align*}
To find when $\phi : \rn^2 \to \sn^n$ is pseudo-umbilical we need only consider the direction of the mean curvature.
\begin{align*}
& \la (\nabla d\phi)(\partial x,\partial x) , H\ra  = \frac{|H|}{2} \sum_{k} R_k\Big(\tfrac{\lambda_1}{4}(\mu_k -\bar{\mu}_k )^2 +1\Big)
- \frac{|H|}{2}\sum_{j}  R'_j\left(\tfrac{\lambda_2}{4}(\eta_j -\bar{\eta}_j )^2 +1\right)\\
& = \frac{|H|}{8} \Big(\lambda_1\sum_{k} R_k(\mu_k -\bar{\mu}_k )^2 -\lambda_2\sum_{j}  R'_j(\eta_j -\bar{\eta}_j )^2 \Big) ,
\end{align*}
and 
\begin{align*}
& \la (\nabla d\phi)(\partial y,\partial y) , H\ra =  \frac{|H|}{2} \sum_{k} R_k\Big(-\tfrac{\lambda_1}{4}(\mu_k +\bar{\mu}_k )^2 +1\Big)
- \frac{|H|}{2}\sum_{j}  R'_j\Big(-\tfrac{\lambda_2}{4}(\eta_j +\bar{\eta}_j )^2 +1\Big)\\
& = \frac{|H|}{8} \Big(-\lambda_1\sum_{k} R_k(\mu_k +\bar{\mu}_k )^2 + \lambda_2\sum_{j}  R'_j(\eta_j +\bar{\eta}_j )^2 \Big) .
\end{align*}
The requirement for pseudo-umbilic is:
\begin{align*}
&\lambda_1\sum_{k} R_k(\mu_k -\bar{\mu}_k )^2 -\lambda_2\sum_{j}  R'_j(\eta_j -\bar{\eta}_j )^2 =
-\lambda_1\sum_{k} R_k(\mu_k +\bar{\mu}_k )^2 + \lambda_2\sum_{j}  R'_j(\eta_j +\bar{\eta}_j )^2 ,
\end{align*}
or equivalently
\begin{align}\label{eqq1}
& \lambda_1\sum_{k} R_k(\mu^2_k +\bar{\mu}^2_k ) -\lambda_2\sum_{j}  R'_j(\eta^2_j +\bar{\eta}^2_j ) = 0 .
\end{align}
The last term is
$$(\nabla d\phi)(\partial x,\partial y) = \frac{\partial^2 \psi}{\partial x \partial y} ,$$
so
\begin{align*}
&(\nabla d\phi)(\partial x,\partial y) =
\tfrac{\i\lambda_1}{4\sqrt{2}} \sum_{k} \sqrt{R_k}\Big(  e^{\tfrac{\sqrt{\lambda_1}}{2}(\mu_k z - \bar{\mu}_k\bar{z})}
 \Big(\mu^2_k -\bar{\mu}^2_k\Big) Z_k
+ e^{\tfrac{\sqrt{\lambda_1}}{2}(-\mu_k z + \bar{\mu}_k\bar{z})}\Big(\mu^2_k -\bar{\mu}^2_k\Big)\bar{Z}_k\Big) \\
&+ \tfrac{\i\lambda_2}{4\sqrt{2}} \sum_{j} \sqrt{R'_j} \Big( e^{\tfrac{\sqrt{\lambda_2}}{2}(\eta_j z - \bar{\eta}_j \bar{z})}
\Big(\eta^2_j -\bar{\eta}^2_j\Big) W_j
+ e^{\tfrac{\sqrt{\lambda_2}}{2}(-\eta_j z + \bar{\eta}_j \bar{z})} \Big(\eta^2_j -\bar{\eta}^2_j\Big)  \overline{W}_j\Big) ,
\end{align*}
so
\begin{align*}
& \la (\nabla d\phi)(\partial x,\partial y) , H\ra = 
\tfrac{\i |H|}{8} \Big(\lambda_1\sum_{k} R_k\Big(\mu^2_k -\bar{\mu}^2_k\Big) 
- \lambda_2\sum_{j}  R'_j\Big(\eta^2_j -\bar{\eta}^2_j\Big) \Big) ,
\end{align*}
and, if $\phi$ is pseudo-umbilic, then 
\begin{align}\label{eqq2}
&\lambda_1\sum_{k} R_k\Big(\mu^2_k -\bar{\mu}^2_k\Big) 
- \lambda_2\sum_{j}  R'_j\Big(\eta^2_j -\bar{\eta}^2_j\Big) =0 .
\end{align}
From Equations~\eqref{eqq1} and~\eqref{eqq2}, we obtain
$$\lambda_1\sum_{k} R_k \mu^2_k - \lambda_2\sum_{j}  R'_j \eta^2_j =0,$$
but
$$\lambda_1\sum_{k} R_k \mu^2_k + \lambda_2\sum_{j}  R'_j \eta^2_j =0 ,$$
so 
$$\lambda_1\sum_{k} R_k \mu^2_k = \lambda_2\sum_{j}  R'_j \eta^2_j =0.$$
Furthermore, it is not difficult to check that the conditions $\lambda_1\sum_{k} R_k \mu^2_k= \lambda_2\sum_{j}  R'_j \eta^2_j=0$ are equivalent with $\phi_{t_1}:\rn^2\to \sn^{2m-1}\left(\tfrac{1}{\sqrt{2}}\right)$, $\phi_{t_2}:\rn^2\to \sn^{2m'-1}\left(\tfrac{1}{\sqrt{2}}\right)$ are harmonic maps and $\phi_{t_1}^* \la , \ra = \tfrac{1-h}{2}\la ,\ra$, $\phi_{t_2}^* \la , \ra = \tfrac{1+h}{2}\la ,\ra$.
\end{proof}

In low dimension, the condition of pseudo-umbilical forbids some combinations of the parameters.

\begin{corollary}
Let $\phi:\rn^2\to \sn^7$ be a CMC proper-biharmonic immersion. Then
$\phi$ is pseudo-umbilical if and only if $m=m'=2$, $R_1=R_2=1/2$, $\mu^2_2=-\mu_1^2$ and
$R'_1=R'_2=1/2$, $\eta^2_2=-\eta^2_1$.
\end{corollary}

\subsection{CMC proper-biharmonic flat surfaces in $\sn^5$}

In dimension five, one can obtain formulas linking the different parameters together, highlighting the structure of the set of flat proper-biharmonic CMC surfaces in $\sn^5$.

\begin{theorem}[Structure theorem] \label{thm6}
For a given $h\in (0,1)$ there is a one-parameter family of CMC proper-biharmonic surfaces $\phi_{h,\rho}=\phi_{\rho}:\rn^2\to\mathbb{S}^5$ with mean curvature $h$,
$\rho\in[0,\tfrac{1}{2}\arccos \tfrac{h-1}{1+h}]$, such that $\psi_{\rho}=i\circ\phi_{\rho}:\rn^2 \to \rn^6$ can be written as
\begin{align*}
\psi_{\rho}(z) = & \frac{1}{\sqrt{2}}\Big( e^{\tfrac{\sqrt{\lambda_1}}{2}(z-\bar{z})}Z_1
+e^{\tfrac{\sqrt{\lambda_1}}{2}(-z+\bar{z})}\bar{Z}_1\Big) \\
& +\frac{1}{\sqrt{2}} \sum_{j=1}^2\sqrt{R'_j}\Big(e^{\tfrac{\sqrt{\lambda_2}}{2}(\eta_j z-\bar{\eta}_j\bar{z})}W_j
+ e^{\tfrac{\sqrt{\lambda_2}}{2}(-\eta_j z+\bar{\eta}_j
\bar{z})}\bar{W}_j\Big) ,
\end{align*}
where
\begin{enumerate}
\item[a)] $Z_1=\frac{1}{2}\Big(E_1-\i E_2\Big)$, \\
\item[b)] $W_j=\frac{1}{2}\Big(E_{2(1+j)-1}-\i E_{2(1+j)}\Big)$, \ $j=1,2$, \\
\item[c)] $\{E_1,\ldots,E_6\}$ is an orthonormal basis of $\rn^6$, \\
\item[d)] $\lambda_1=2(1-h)$, $\lambda_2=2(1+h)$,
\end{enumerate}
and $R'_1$, $R'_2$, $\eta_1=e^{\i\rho}$ and $\eta_2=e^{\i\tilde{\rho}}$ are given by
$$
\left(\frac{1-\Big(\frac{1-h}{h+1}\Big)^2}{2\Big(1+ \frac{1-h}{h+1}\cos 2\rho\Big)}, \ 1-\frac{1-\Big(\frac{1-h}{h+1}\Big)^2}{2\Big(1+ \frac{1-h}{h+1}\cos 2\rho\Big)}, \ \rho, \ \tilde{\rho}=\arctan{(-\frac{1}{h\tan{\rho}})} \right),
$$
if $\rho\in(0,\tfrac{1}{2}\arccos \tfrac{h-1}{1+h}]$, and
$$
\Big(\frac{h}{1+h}, \ \frac{1}{1+h}, \ 0, \ -\frac{\pi}{2}\Big),
$$
if $\rho=0$.

Conversely, assume that $\phi:\rn^2\to\mathbb{S}^5$ is a CMC proper-biharmonic surface with mean curvature $h\in (0,1)$. Then, up to isometries of $\mathbb{R}^2$ and $\mathbb{R}^6$, $\psi=i\circ\phi:\rn^2 \to \rn^6$ is one of the above maps.
\end{theorem}

\begin{proof}

Let a map $\phi : \rn^2 \to \sn^5$ be given by Formula~\eqref{formula*}, we can assume wlog that $m=1$, $m'=2$, $R_1=1$ and put $R_1' = s$ and $R_2'=1-s$ 
($s\in (0,1)$).

\begin{lemma} \label{lemma1}
Let $\eta_1 = e^{\i \rho}$ with, because of symmetries of solutions, $\rho \in [0,\tfrac{\pi}{2}]$ and $\eta_2 = e^{\i \tilde{\rho}}$, $\tilde{\rho} \in [-\tfrac{\pi}{2} , \tfrac{\pi}{2})$, then 
\begin{equation*}
\tilde{\rho} = 
\begin{dcases}
-\tfrac{\pi}{2} ,\quad \mbox{ if } \quad \rho=0 , \\
0, \quad \mbox{if} \quad\rho= \tfrac{\pi}{2}, \\
\arctan\left( \frac{-1}{h\tan\rho}\right) ,\quad \mbox{ otherwise.}
\end{dcases}
\end{equation*}
In particular, $\tilde{\rho} \in [-\tfrac{\pi}{2} , 0]$.
\end{lemma}

\begin{proof}
Consider $\mu_1 = \cos\omega_1 + \i \sin\omega_1$ and $z=x+\i y$, then 
\begin{align*}
& e^{\tfrac{\sqrt{\lambda_1}}{2}\left( \mu_1 z - \bar{\mu}_1\bar{z}\right)} Z_1 + 
e^{\tfrac{\sqrt{\lambda_1}}{2}\left( -\mu_1 z + \bar{\mu}_1\bar{z}\right)} \bar{Z}_1 = \\
& \cos\left( \sqrt{\lambda_1}( y\cos\omega_1 + x \sin \omega_1 )\right) E_1 + \sin\left( \sqrt{\lambda_1}( y\cos\omega_1 + x \sin \omega_1 )\right) E_2 .
\end{align*}
We consider now, for a given $h\in (0,1)$,
$$ ( 1  , R'_1 , R'_2 , \mu_1, \eta_1 ,\eta_2 )$$
a solution of (e)--(h) and $\psi : \rn^2 \to \rn^6$ the corresponding CMC proper-biharmonic immersion. Let $T_\alpha : \rn^2 \to \rn^2$ be the rotation of angle $\alpha$. Since the first component of $\psi(T_{\alpha}(z))$ is given by
$$ \cos\left( \sqrt{\lambda_1}(\sin (\alpha + \omega_1) x + \cos (\alpha+\omega_1) y) \right),$$
for $\alpha=-\omega_1$, the composition $\psi\circ T_{-\omega_1}$ is the CMC proper-biharmonic immersion corresponding to 
$$( 1, R'_1 , R'_2 , 1, \mu_1^{-1}\eta_1 ,\mu_1^{-1}\eta_2 )= ( 1, R'_1 , R'_2 , 1 , \bar{\mu_1}\eta_1 ,\bar{\mu_1}\eta_2 ).$$
We re-denote $\bar{\mu_1}\eta_1$ and $\bar{\mu_1}\eta_2$ by $\eta_1= e^{\i\rho}$ and $\eta_2= e^{\i\tilde\rho}$ and we can assume that 
$\rho,\tilde\rho \in [-\tfrac{\pi}{2} , \tfrac{3\pi}{2})$

If we consider the isometry of $\rn^6$, $T(E_j)= E_j$ for $j=1,\dots,6$ except $T(E_4)=-E_4$, then $T\circ\psi$ corresponds to the data
$$( 1 , R'_1 , R'_2 , 1 , -\eta_1 ,\eta_2 )= ( 1 , R'_1 , R'_2 , 1 , e^{\i(\rho+ \pi)} , e^{\i\tilde\rho}).$$
Therefore we re-denote $-\eta_1$ with $\eta_1$ and can assume that $\rho\in [-\tfrac{\pi}{2},\tfrac{\pi}{2})$ and, similarly, 
$\tilde\rho\in [-\tfrac{\pi}{2},\tfrac{\pi}{2})$.

Furthermore, composition of $\psi$ with the symmetry with respect to the $Oy$-axis yields a CMC proper-biharmonic immersion corresponding to the data
$$( 1 , R'_1 , R'_2 , 1 , \bar{\eta_1} ,\bar{\eta_2} )= ( 1  , R'_1 , R'_2 , 1, e^{-\i\rho} , e^{-\i\tilde\rho}),$$
so we can restrict the angles to $\rho\in [0,\tfrac{\pi}{2}]$ and $\tilde\rho\in [-\tfrac{\pi}{2},\tfrac{\pi}{2})$.

Since 
\begin{equation}\label{eq1}
(1-h) + (1+h)s\eta^2_1  + (1+h)(1-s)\eta^2_2 =0 ,
\end{equation}
we have
$$\eta^2_2= \frac{1}{1-s}\Big( \frac{h-1}{h+1} -s\eta^2_1\Big).$$
Since $\eta_2$ is a unit complex number
$$1-s = \Big| \frac{1-h}{h+1} + \eta^2_1 s\Big| = \Big|\frac{1-h}{h+1} +  s\cos 2\rho + \i s \sin 2\rho\Big| .$$
Equivalently
\begin{align*}
(1-s)^2 &=  \Big(\frac{1-h}{h+1}\Big)^2 + s^2 + 2 s\frac{1-h}{h+1} \cos 2\rho ,
\end{align*}
hence
\begin{align*}
s &= \frac{1 - \Big( \tfrac{1-h}{h+1}\Big)^2}{2\Big( 1 + \frac{1-h}{h+1} \cos 2\rho \Big)} = \frac{2h}{(1+h)(1+h + (1-h)\cos 2\rho)}.
\end{align*}
Then, from Equation~\eqref{eq1}, the real part of $\eta_2^2$ is
\begin{align}\label{eq2}
\cos 2\tilde{\rho} &= \frac{h^2 -1 - \cos 2\rho (1+h^2)}{1 + h^2 + (1-h^2) \cos 2\rho}
\end{align}
and its imaginary part is
\begin{align}\label{eq3}
\sin 2\tilde{\rho} &= \frac{-2h}{1 + h^2 + (1-h^2) \cos 2\rho}\sin 2\rho .
\end{align}
Assume that the real part of $\eta_2^2$ is zero then 
$$ \cos 2\rho = \frac{h^2 -1}{h^2+1} \quad \mbox{ and } \quad \sin 2\rho= \frac{2h}{h^2+1},$$
therefore 
$$s= \frac{1+h^2}{(h+1)^2},$$ 
and, using~\eqref{eq1}, we have:
\begin{align*}
\cos 2\tilde{\rho} + \i (1+ \sin 2 \tilde{\rho})=0, \\
\end{align*}
thus $\tilde{\rho} = -\tfrac{\pi}{4}$. Note that $\cos\rho=\tfrac{h}{\sqrt{h^2 + 1}}$ and 
$\sin\rho=\tfrac{1}{\sqrt{h^2 + 1}}$, so $\tan\rho = \tfrac{1}{h}$ and $\tfrac{-1}{h\tan\rho} = -1 = \tan \tilde\rho$.

Assume now that the imaginary part of $\eta_2^2$ vanishes, i.e. $\sin 2\tilde\rho =\sin 2\rho =0$. Then either $\cos 2\rho = 1$ and then $\eta_1^2 =1$, $\eta_2^2=-1$ and $s=\tfrac{h}{h+1}$, so $\rho =0$ and $\tilde\rho =-\tfrac{\pi}{2}$, whilst if $\cos 2\rho = -1$ then $\eta_1^2 =-1$, $\eta_2^2=1$ with $s=\tfrac{1}{h+1}$, so $\rho =\tfrac{\pi}{2}$ and $\tilde\rho =0$.

Finally, if neither the real nor imaginary part of $\eta_2^2$ vanishes, as 
$$\tan 2\tilde\rho = \frac{2 \tan\tilde\rho}{1- \tan^2\tilde\rho} ,$$
so 
$$\tan\tilde\rho= - \frac{\cos 2\tilde{\rho} +1}{\sin 2\tilde{\rho}} \quad  
\mbox{ or } \quad \tan\tilde\rho= \frac{-\cos 2\tilde{\rho} +1}{\sin 2\tilde{\rho}} ,$$
therefore, by Equations~\eqref{eq2} and \eqref{eq3}:
$$\tan\tilde\rho= h \frac{1-\cos 2\rho}{\sin2\rho}= h\tan \rho ,$$ 
or 
$$\tan\tilde\rho= - \frac{1+\cos 2\rho}{h\sin2\rho}= \frac{-1}{h\tan \rho}.$$

We can eliminate the first possibility by checking Equations~\eqref{eq2} and \eqref{eq3}. 
If $\tan\tilde\rho= h\tan \rho$ then, since $\cos2\rho =\frac{1-\tan^2 \rho}{1+\tan^2 \rho}$, 
$$\cos2\tilde\rho = \frac{1-h^2\tan^2 \rho}{1+h^2\tan^2 \rho},$$
but Equation~\eqref{eq2} implies
$$\cos2\tilde\rho =  \frac{(h^2 -1)(1+\tan^2 \rho) - (h^2 +1)(1-\tan^2 \rho)}{(h^2 +1)(1+\tan^2 \rho) + (1-h^2)(1-\tan^2 \rho)},$$
which forces $\tan\rho = \tfrac{1}{h}$, hence $\tan\tilde\rho = 1$. From the imaginary part we have, on the one hand:
$$\sin2\tilde\rho = \frac{2\tan\tilde\rho}{1+\tan^2 \tilde\rho}=1,$$
while, with $\sin2\rho =\frac{2\tan\rho}{1+\tan^2 \rho}$, Equation~\eqref{eq3} shows that
\begin{align*}
\sin2\tilde\rho &= \frac{-2h \tfrac{2\tan\rho}{1+\tan^2\rho}}{1+h^2 + (1-h^2)\tfrac{1-\tan^2\rho}{1+ \tan^2\rho}} = -1 ,
\end{align*}
so we have a contradiction, which does not happen for the second solution $\tan\tilde\rho = \tfrac{-1}{h\tan \rho}$.

Let $\psi : \rn^2 \to \rn^6$ be the CMC proper-biharmonic immersion corresponding to
$$( 1 ,  R'_1 , R'_2 , 1 , \eta_1 ,\eta_2 )= ( 1 , R'_1 , R'_2 , 1 ,e^{\i\rho} , e^{\i\tilde\rho}),$$
where $\rho\in [0,\tfrac{\pi}{2}]$ and $\tilde\rho\in [-\tfrac{\pi}{2},0]$. We compose $\psi$ with the symmetry of $\rn^6$ defined by
$$ T(E_1) = E_1, \, T(E_2) = E_2, \,T(E_3) = E_5, \,T(E_4) = E_6, \,T(E_5) = E_3, \,T(E_6) = E_4, $$
and the isometry of $\rn^2$, $t(x,y)=(-x,y)$, so that the resulting immersion corresponds to the data
$$( 1 , R'_2 , R'_1 , 1 , \bar{\eta_2} ,\bar{\eta_1} )= ( 1 , R'_2 , R'_1 , 1 , e^{-\i\tilde\rho} , e^{-\i\rho}).$$
Therefore, we can assume that $0<R'_1 \leq \tfrac{1}{2}$ and $\rho\in [0,\tfrac{1}{2}\arccos \tfrac{h-1}{h+1}]$.

\end{proof}

The proof of Theorem~\ref{thm6} then easily follows from Lemma~\ref{lemma1}.

\end{proof}

\begin{remark} 
If $\rho=\frac{1}{2}\arccos\frac{h-1}{1+h}$ then the solution
$$
(R'_1, \ R'_1, \ \rho, \ \tilde{\rho})
$$
of \eqref{eq1} is
$$
\Big( \frac{1}{2}, \ \frac{1}{2}, \ \frac{1}{2}\arccos\frac{h-1}{1+h}, \ -\frac{1}{2}\arccos\frac{h-1}{1+h} \Big).
$$
\end{remark}

A useful result which will be used further on is the following lemma.

\begin{lemma}\label{lemma2}
Let $t=\tan \rho/2$, $\rho\in[0,(\pi/2)]$, then $\eta_1 = e^{\i \rho}$ and $\eta_2 = e^{\i \tilde\rho}$ are solutions of 
$$ (1-h) + (1+h)s\eta^2_1  + (1+h)(1-s)\eta^2_2 =0 ,$$
if and only if $s\in[\frac{h}{1+h}, \frac{1}{1+h}]$,
$$ \tan\tilde\rho= \frac{-1}{h\tan \rho} \quad \mbox{when}\quad s\in \left(\frac{h}{1+h}, \frac{1}{1+h}\right) ,$$ 
$\tilde{\rho}=-(\pi/2)$ for $s=h/(h+1)$ and $\tilde{\rho}=0$ for $s=1/(h+1)$,
and
\begin{equation}
t = 
\begin{dcases}
0 \quad \mbox{if} \quad s=\frac{h}{1+h} ,\\
1 \quad \mbox{if} \quad s=\frac{1}{1+h} ,\\
\frac{\sqrt{s(1-h^2)} - \sqrt{h(1-s-hs)}}{\sqrt{s-(1-s)h}} \quad \mbox{otherwise.}
\end{dcases}
\end{equation}
\end{lemma}

\begin{proof}
From Lemma~\ref{lemma1} we know that
\begin{equation}\label{eq1lemma2} 
s = \frac{2h}{(1+h)(1+h + (1-h)\cos 2\rho)},
\end{equation}
with $\rho \in [0,\tfrac{\pi}{2}]$. The function $s=s(\rho)$ is strictly increasing on $[0,\tfrac{\pi}{2}]$ and since $s(0)= \tfrac{h}{h+1}$, when $(\rho,\tilde\rho)=(0,-\tfrac{\pi}{2})$, and $s(\tfrac{\pi}{2})= \tfrac{1}{h+1}$, for $(\rho,\tilde\rho)=(\tfrac{\pi}{2},0)$, we have $s\in [\tfrac{h}{h+1},\tfrac{1}{h+1}]$.

Assume $s\in \left(\tfrac{h}{h+1},\tfrac{1}{h+1}\right)$ then, working with the variable $t=\tan \rho/2$, so that $t\in (0,1)$, and
$$\cos 2\rho = \frac{1-6t^2 + t^4}{(1+t^2)^2},$$ 
we obtain the equation
$$ t^4 -2 a t^2 +1 =0,$$
where $a = \tfrac{s+h-hs-2sh^2}{s-(1-s)h}$ is strictly greater than $1$. The only admissible solution is $t^2 = a - \sqrt{a^2 -1}$, i.e.
$$ t = \frac{\sqrt{s(1-h^2)} - \sqrt{h(1-s-hs)}}{\sqrt{s-(1-s)h}}.$$
\end{proof}

An immediate consequence of Theorem~\ref{thm6} is the existence, for any $h\in (0,1)$ of proper-biharmonic planes in $\sn^5$ with CMC equal to $h$.

\begin{corollary}[Existence on $\rn^2$]\label{cor2}
Let $h\in (0,1)$, then there exist proper-biharmonic immersions of $\rn^2$ into $\sn^5$ with CMC equal to $h$.
\end{corollary}

In order to quotient a CMC proper-biharmonic immersion $\phi_{h,\rho} = \phi : \rn^2 \to \sn^5$ to a cylinder or a torus, we start with some remarks. We observe first that if we denote $\frac{\sqrt{\lambda_2}}{2} (\eta_j z - \bar{\eta_j}\bar{z}) = \i \theta_j$, $\theta_j =\theta_j (x,y)$ real number, then 
$\theta_j = \langle \i \sqrt{\lambda_2}\bar{\eta_j} , z \rangle$, where $\langle , \rangle$ denotes the usual inner product on $\rn^2$. We have
$$e^{\tfrac{\sqrt{\lambda_2}}{2}(\eta_j z - \bar{\eta}_j \bar{z})}W_j
+ e^{\tfrac{\sqrt{\lambda_2}}{2}(-\eta_j z + \bar{\eta}_j \bar{z})}\overline{W}_j = \cos\theta_j E_{2(1+j) -1} + \sin\theta_j E_{2(1+j)},$$
and from here we get that $\psi(z_1) = \psi(z_2)$ is equivalent to $\psi(z_1 -z_2) = \psi(0)$. We define
\begin{align*}
\Lambda_{\psi} &= \left\{ z\in \rn^2 : \psi(z) = \psi(0)\right\} \\
&= \left\{ z\in \rn^2 : \langle \i\sqrt{\lambda_1},z\rangle \equiv  \langle \i\sqrt{\lambda_2}\bar{\eta_1},z\rangle \equiv \langle \i\sqrt{\lambda_2}\bar{\eta_2},z\rangle \equiv 0 \pmod{2\pi}\right\}
\end{align*}

We observe that $\Lambda_{\psi}$ is a discrete lattice and the map $\psi$ quotients to an injective map, also denoted by $\psi$, $\psi : \rn^2 /\Lambda_{\psi} \to \rn^6$ or $\phi : \rn^2 /\Lambda_{\psi} \to \sn^5$. In the following, as we are interested in a existence result, we will assume that $\rho\in [0,\tfrac{\pi}{2})$.

\begin{proposition}\label{prop10}
Let $\phi_{h,\rho} = \phi_\rho : \rn^2 \to \sn^5$ be a CMC proper-biharmonic immersion with $|H|=h\in (0,1)$ and $\rho\in [0,\tfrac{\pi}{2})$. We have
\begin{itemize}
\item[i)] If $v= K_2 \left(\tfrac{2\pi}{\sqrt{\lambda_2}}, 0\right)$, $K_2\in \zn$, then $\phi_0 (0) = \phi_0 (v)$.
\item[ii)] If $v= \Big( \tfrac{2\pi}{\sin \rho}\Big( \tfrac{K_1}{\sqrt{\lambda_2}} - \tfrac{K_0}{\sqrt{\lambda_1}}\cos\rho \Big) , \tfrac{2\pi}{\sqrt{\lambda_1}} K_0 \Big)$,
where $K_0 ,K_1 \in \zn$ such that $|K_1 - \sqrt{\tfrac{\lambda_2}{\lambda_1}} K_0| >0$ and $\rho\in (0,\tfrac{\pi}{2})$ such that
$$ \frac{\sin\tilde\rho}{\sin\rho}\Big( K_1 - \sqrt{\tfrac{\lambda_2}{\lambda_1}} K_0 \cos\rho\Big) + \sqrt{\tfrac{\lambda_2}{\lambda_1}} K_0\cos\tilde\rho \in \zn ,$$
then $\phi_{\rho} (0) = \phi_{\rho} (v)$.
\end{itemize}
\end{proposition}

\begin{proof}
We know that $\phi_{\rho} (0) = \phi_{\rho} (v)$ if and only if
$$\langle \i\sqrt{\lambda_1},v\rangle \equiv  \langle \i\sqrt{\lambda_2}\bar{\eta_1},v\rangle \equiv \langle \i\sqrt{\lambda_2}\bar{\eta_2},v\rangle \equiv 0 \pmod{2\pi},$$
and, if $v=(T,V)$
\begin{align}
\langle \i\sqrt{\lambda_1},v\rangle &= 2\pi K_0 \Leftrightarrow V= \tfrac{2\pi}{\sqrt{\lambda_1}} K_0 , \notag\\
\langle \i\sqrt{\lambda_2}\bar{\eta_1},v\rangle &= 2\pi K_1 \Leftrightarrow T \sin\rho  = 2\pi\left( \tfrac{K_1}{\sqrt{\lambda_2}} - \tfrac{K_0}{\sqrt{\lambda_1}}\cos\rho \right) , \label{condition1}\\
\langle\i\sqrt{\lambda_2}\bar{\eta_2},v\rangle &= 2\pi K_2 \Leftrightarrow T \sin\tilde\rho  = 2\pi\left( \tfrac{K_2}{\sqrt{\lambda_2}} - \tfrac{K_0}{\sqrt{\lambda_1}}\cos\tilde\rho \right) . \label{condition2}
\end{align}
Case i): Assume now that $\rho=0$, and therefore $\tilde\rho = -\frac{\pi}{2}$, so \eqref{condition1} becomes $K_1 = K_0 \sqrt{\tfrac{\lambda_2}{\lambda_1}}$ and 
\eqref{condition2} is $T= - \tfrac{2\pi K_2}{\sqrt{\lambda_2}}$. If $K_0=0$ then $K_1=0$ and $v$ can be chosen to be 
$$v= \left( - \tfrac{2\pi K_2}{\sqrt{\lambda_2}} , 0 \right), \quad K_2 \in \zn.$$
Case ii): Assume that $\rho \in (0,\tfrac{\pi}{2})$. Then Condition~\eqref{condition1} becomes
$$ T = \frac{2\pi}{\sin\rho} \left( \tfrac{K_1}{\sqrt{\lambda_2}} - \tfrac{K_0}{\sqrt{\lambda_1}}\cos\rho \right)  ,$$
and Condition~\eqref{condition2} yields
$$ K_2 = \frac{\sin\tilde\rho}{\sin\rho} \left( K_1 - \sqrt{\frac{\lambda_2}{\lambda_1}} K_0 \cos\rho \right) + \sqrt{\frac{\lambda_2}{\lambda_1}}K_0 \cos\tilde\rho .$$
From $\tan \tilde\rho = \frac{-1}{h\tan\rho}$, we get that $\rho$ goes to $0$ if and only if $\tilde\rho$ goes to $-\tfrac{\pi}{2}$ and $\rho$ goes to $\tfrac{\pi}{2}$ if and only if $\tilde\rho$ goes to $0$.
Let 
$$ F(\rho) = \frac{\sin\tilde\rho}{\sin\rho} \left( K_1 - \sqrt{\frac{\lambda_2}{\lambda_1}} K_0 \cos\rho \right) + \sqrt{\frac{\lambda_2}{\lambda_1}}K_0 \cos\tilde\rho .$$ 
Assume that $K_0$ and $K_1$ satisfy $\left|K_1 - \sqrt{\tfrac{\lambda_2}{\lambda_1}} K_0\right| >0$. We have 
\begin{align*}
\lim_{\rho\to 0} F(\rho) &= - \infty \sgn\left(K_1 - \sqrt{\tfrac{\lambda_2}{\lambda_1}} K_0\right) \\
\lim_{\rho\to \tfrac{\pi}{2}} F(\rho) &=  \sqrt{\tfrac{\lambda_2}{\lambda_1}} K_0 .
\end{align*}
Therefore, there exists an infinite number of $\rho\in (0,\tfrac{\pi}{2})$ such that $F(\rho)\in\zn$.
\end{proof}

A mere examination of the condition required for periodicity of solutions, Proposition~\ref{prop10}, shows that it will always be possible to find cylinders of CMC $h$ in $\sn^5$, for any value of $h$ in $(0,1)$.

\begin{corollary}[Existence on cylinders]\label{cor3}
Let $h\in (0,1)$, then there exist proper-biharmonic cylinders in $\sn^5$ with CMC equal to $h$.
\end{corollary}

Let 
$$\phi_{h,0} : \rn^2 \to \sn^5$$
be the map given by Formula~\eqref{formula*} with $\rho=0$ and $\tilde\rho= -\tfrac{\pi}{2}$. Then the lattice 
$\Lambda_{\psi_{h,0}} = \{ z\in \rn^2 : \psi_{h,0}(z) = \psi_{h,0}(0)\}$ on which $\psi_{h,0}$ is constant, is equal to the lattice
$$\left\{ m v_2 + n v_1 : m,n\in \zn \quad \mbox{s.t.}\quad m \sqrt{\tfrac{\lambda_2}{\lambda_1}} \in \zn\right\} ,$$
where $v_1 = \left( - \tfrac{2\pi}{\sqrt{\lambda_2}} , 0\right)$ and $v_2 = \left( 0 ,  \tfrac{2\pi}{\sqrt{\lambda_1}} \right)$.

If $\sqrt{\tfrac{\lambda_2}{\lambda_1}} \in \rn\setminus \qn$, then the rank of $\Lambda_{\psi_{h,0}}$ is one and the quotient map $\phi_{h,0}$ from the cylinder 
$\rn^2/\Lambda_{\psi_{h,0}}$ into $\sn^5$ is injective. If $\sqrt{\tfrac{\lambda_2}{\lambda_1}} = \tfrac{a}{b}$ then
$$\Lambda_{\psi_{h,0}} = \left\{ m (bv_2) + n v_1 : m,n\in \zn \right\} ,$$
has rank two.

\begin{proposition}\label{prop?}
The CMC proper-biharmonic immersion $\phi_{h,0} : \rn^2 \to \sn^5$ quotients to a cylinder for any value of $h\in (0,1)$. Moreover, $\phi_{h,0}$ is an injective map from a cylinder to $\sn^5$ if and only if $\sqrt{\tfrac{\lambda_2}{\lambda_1}}$ is irrational and it quotients to a
torus if and only if $\sqrt{\tfrac{\lambda_2}{\lambda_1}}$  is rational.
\end{proposition}

\begin{remark}\label{rmk4}
It is easy to see that $\sqrt{\tfrac{\lambda_2}{\lambda_1}}$ is rational if and only if
$h=\frac{q^2-1}{q^2+1}$, where $q$ is a rational number strictly greater than
$1$, and the set $\{ \frac{q^2-1}{q^2+1} : q\in \qn, q>1\}$ is dense in $(0,1)$.
\end{remark}

Though double periodicity, and therefore quotients to tori, is not automatic, as for cylinders, we exhibit a condition on the value of $h$, in terms of a pair of rational numbers, which is equivalent to the existence of proper-biharmonic tori of CMC $h$ in $\sn^5$. Note that this formula rules out the possibility of proper-biharmonic tori with irrational CMC. In particular, the set of admissible values of $h$ is infinite countable, without any gap.

\begin{proposition}\label{prop??}
The CMC proper-biharmonic immersion $\phi_{h,\rho} : \rn^2 \to \sn^5$, $\rho\in (0,\tfrac{\pi}{2})$, quotients to a torus if and only if
$$ h = \frac{1- (a-b)^2}{1 + (a-b)^2 + 2(a+b)} ,$$
where $a=p^2/q^2$ and $b=r^2/t^2$ with $p,q,r,t \in \nn^*$, under the condition $0\leq (b-a)^2 <1$. Moreover, in this case
$$\Lambda_{\psi_{h,\rho}} = \{ m v_2 + n v_1 \, : \, m, n \in \zn \mbox{ \upshape{s.t.} } m \tfrac{q}{p} - n \tfrac{qr}{pt} \in \zn\},$$
where 
$$v_1 = \left( \frac{\pi\sqrt{(a-b)^2 + a + b}}{\sqrt{a}}, 0\right), $$
and
$$v_2 = \left( \frac{-\pi\sqrt{b/a}(1-(a-b))}{\sqrt{(a-b)^2 + a + b}}  , \pi\sqrt{\frac{(a-b)^2 + 2(a+b) +1}{(a-b)^2 + a + b}} \right).$$
\end{proposition}

\begin{proof}
Let $v=(T,V)$ be a (non-zero) vector of $\rn^2$. Then $\phi$ is periodic in the direction of $v$ if, as in the previous corollary:
$$V=\frac{2\pi K_0}{\sqrt{\lambda_1}}, \quad T= \frac{2\pi}{\sqrt{\lambda_2}\sin{\rho}} \Big( K_1 - \sqrt{\frac{\lambda_2}{\lambda_1}} K_0 \cos\rho \Big),$$
and
$$\sqrt{\lambda_2} T \sin{\tilde\rho} \, + \sqrt{\lambda_2} V \cos\tilde\rho = 2\pi K_2 ,$$
with $K_0 , K_1 , K_2 \in \zn$.\\
Similarly if $\tilde v =(\widetilde T,\widetilde V)$ is another (non-zero) vector of $\rn^2$, then $\phi$ is periodic in the direction of $\tilde v$ if, as in the previous corollary:
$$\widetilde V=\frac{2\pi \widetilde K_0}{\sqrt{\lambda_1}}, \quad \widetilde T= \frac{2\pi}{\sqrt{\lambda_2}\sin{\rho}} \Big( \widetilde K_1 - \sqrt{\frac{\lambda_2}{\lambda_1}} \widetilde K_0 \cos\rho  \Big),$$
and
$$\sqrt{\lambda_2} \widetilde T \sin{\tilde\rho} \, + \sqrt{\lambda_2} \widetilde V \cos\tilde\rho \, = 2 \pi \widetilde K_2,$$
with $\widetilde K_0 , \widetilde K_1 , \widetilde K_2 \in \zn$.\\ 
If $t=\tan \rho/2$ then $t\in (0,1)$ and $\cos \rho = \frac{1-t^2}{1+t^2}$ and $\sin\rho = \frac{2t}{1+t^2}$. 
Then 
$$\cos 2\rho = \frac{1-6t^2 + t^4}{(1+t^2)^2}$$ 
and
$$\cos 2\tilde\rho = \frac{-1+ (2+4h^2)t^2 - t^4}{1+(-2+ 4h^2)t^2 + t^4},$$
so 
$$ \cos \tilde\rho = \frac{2ht}{\sqrt{4h^2 t^2 + (1-t^2)^2}}$$
and
$$\sin \tilde\rho = \frac{t^2 -1}{\sqrt{4h^2 t^2 + (1-t^2)^2}}.$$
Then the system becomes
\begin{equation}\label{eq**}
\begin{cases}
A \widetilde K_0 + B \widetilde K_1 = \widetilde K_2 \\
A K_0 + B K_1 = K_2 ,
\end{cases}
\end{equation}
with $K_0 , K_1 , K_2 , \widetilde K_0 , \widetilde K_1 , \widetilde K_2 \in \zn$, where 
$$A = \sqrt{\frac{\lambda_2}{\lambda_1}} \frac{1}{\sqrt{4h^2 t^2 + (1-t^2)^2}} \left(\frac{(1-t^2)^2}{2t} + 2ht\right) ,$$
and
$$B= \frac{t^4 -1}{2t \sqrt{4h^2t^2 + (1-t^2)^2}} ,$$
with $A>0$ and $B<0$.

Observe that $v = \lambda \tilde v$ is equivalent to 
\begin{equation*}
\begin{cases}
K_0 = \lambda \widetilde K_0 ,\\
K_1 = \lambda \widetilde K_1 ,
\end{cases}
\end{equation*}
and if $A$ and $B$ are in $\qn$, we can find integers $K_0 , K_1 , K_2 , \widetilde K_0 , \widetilde K_1 , \widetilde K_2$, with $K_0^2+K_1^2>0$,
$\tilde{K}_0^2+\tilde{K}_1^2>0$,  
$(K_0 , K_1) \neq \rn (\widetilde K_0 , \widetilde K_1)$, such that Equation~\eqref{eq**} is satisfied. Conversely, if \eqref{eq**} is satisfied with $(K_0 , K_1) \neq \rn (\widetilde K_0 , \widetilde K_1)$ then the determinant of \eqref{eq**} is non-zero and
\begin{align*}
A &= \frac{\widetilde K_1 K_2 - K_1 \widetilde K_2}{K_0\widetilde K_1 - K_1 \widetilde K_0} , \\
B &= \frac{-\widetilde K_0 K_2 + K_0 \widetilde K_2}{K_0\widetilde K_1 - K_1 \widetilde K_0} .
\end{align*}
So $A$ and $B$ are in $\qn$.

Since $\lambda_1 = 2(1-h)$ and $\lambda_2 = 2(1+h)$ and using Lemma~\eqref{lemma2} we have that
\begin{align*}
A &= \sqrt{\frac{h}{(1-s)(s-(1-s)h)}} ,\\
B &= -\sqrt{\frac{s(1-s-hs)}{(1-s)(s-(1-s)h)}} ,
\end{align*}
and $A$ and $B$ are in $\qn$ if and only if $1/A\in \qn$ and $- B/A \in \qn$ that is
\begin{subnumcases}{}
\frac{(1-s)(s-(1-s)h)}{h} = \frac{p^2}{q^2} \label{sysA} \\
\frac{s(1-s-sh)}{h}  = \frac{r^2}{t^2} \label{sysB}
\end{subnumcases}
with $p,q,r,t \in \nn^*$ and $\frac{h}{h+1} < s < \frac{1}{h+1}$.

Then both $A$ and $B$ are in $\qn$ if and only if there exists $s\in \left(\frac{h}{h+1},\frac{1}{h+1}\right)$ solution to 
\begin{subnumcases}{\label{sys}}
(1+h)s^2 - (2h+1)s +h (1+ a) =0 , \label{sysa} \\
(1+h)s^2 - s + h b =0 , \label{sysb}
\end{subnumcases}
where we put $a=\frac{p^2}{q^2}$ and $b=\frac{r^2}{t^2}$ ($p,q,r,t \in \nn^*$). Assume $s$ is a common solution to Equations~\eqref{sysa} and \eqref{sysb}, by taking the difference between them we obtain that it must be
$$ s= \frac{1}{2} (1 + a -b).$$
Conversely, replacing $s$ by the expression $\frac{1}{2} (1 + a -b)$ in Equation~\eqref{sysa} and Equation~\eqref{sysb} yields the same condition:
$$ h = \frac{1- (a-b)^2}{4b + (1+a-b)^2}.$$
Moreover, the condition $h\in (0,1)$ is equivalent to $0\leq (b-a)^2 <1$ and it automatically ensures that $s\in \left(\frac{h}{h+1},\frac{1}{h+1}\right)$.

Therefore, the only obstruction to solving System~\eqref{sys} is $h = \frac{1- (a-b)^2}{4b + (1+a-b)^2}$, where $a=\frac{p^2}{q^2}$ and $b=\frac{r^2}{t^2}$ ($p,q,r,t \in \nn^*$), with $0\leq (b-a)^2 <1$.

From $\psi_{h,\rho}(0)=\psi_{h,\rho}(v)$, we know that
\begin{align*}
v &= \left( \frac{2\pi K_1}{\sqrt{\lambda_2}\sin\rho} -  \frac{2\pi K_0 \cos\rho }{\sqrt{\lambda_1}\sin\rho} ,  \frac{2\pi K_0}{\sqrt{\lambda_1}}\right) \\
&= K_1 \left( \frac{2\pi}{\sqrt{\lambda_2}\sin\rho} ,0\right) + K_0 \left( -  \frac{2\pi \cos\rho}{\sqrt{\lambda_1}\sin\rho} ,  \frac{2\pi}{\sqrt{\lambda_1}}\right) \\
&= K_1 v_1 + K_0 v_2 ,
\end{align*}
where $v_1 = \left( \tfrac{2\pi}{\sqrt{\lambda_2}\sin\rho} ,0\right)$ and 
$v_2 = \left( -  \tfrac{2\pi \cos\rho}{\sqrt{\lambda_1}\sin\rho} ,  \tfrac{2\pi}{\sqrt{\lambda_1}}\right)$, and $K_0$ and $K_1$ must satisfy 
$$\sqrt{\frac{\lambda_2}{\lambda_1}} \left( \cos\tilde{\rho} - \frac{\sin\tilde\rho}{\tan\rho}\right) K_0 + \frac{\sin\tilde\rho}{\sin\rho} K_1 \in \zn,$$
i.e. $A K_0 + B K_1 \in \zn$.\\
As $\tfrac{1}{A^2} = \tfrac{p^2}{q^2}=a$ and $\tfrac{B^2}{A^2}= \tfrac{r^2}{t^2}=b$, and redenoting $K_0=m$ and $K_1=n$, we obtain
$$\Lambda_{\psi_{h,\rho}} = \{ mv_2 + nv_1 \, : \, m, n \in \zn \mbox{ s.t. } m\tfrac{q}{p} - n\tfrac{qr}{pt} \in \zn\}.$$
Further, as
$$ t = \frac{\sqrt{s(1-h^2)} - \sqrt{h(1-s-sh)}}{\sqrt{s-(1-s)h}} ,$$
and $s= (1/2)(1+a-b)$, $h= \tfrac{1-(a-b)^2}{4b + (1+a-b)^2}$, we obtain
\begin{align*}
t^2 &= \frac{\sqrt{1+a-b}\sqrt{(a-b)^2 + (a+b)} - \sqrt{b}(1-a+b)}{\sqrt{1+a-b}\sqrt{(a-b)^2 + (a+b)} + \sqrt{b}(1-a+b)} ,\\
\sin^2\rho &= \frac{a(1+ (a-b)^2 + 2(a+b))}{(1+a+b)((a-b)^2 + (a+b))} ,\\
\cos^2\rho &= \frac{b(1-a+b)^2}{(1+a+b)((a-b)^2 + (a+b))} ,\\
\sin^2\tilde\rho &= \frac{b(1+ (a-b)^2 + 2(a+b))}{(1+a+b)((a-b)^2 + (a+b))} ,\\
\cos^2\tilde\rho &= \frac{a(1-a+b)^2}{(1+a+b)((a-b)^2 + (a+b))} ,\\
v_1 &= \left( \frac{\pi\sqrt{(a-b)^2 + a + b}}{\sqrt{a}}, 0\right) ,
\end{align*}
and
$$v_2 = \left( \frac{-\pi\sqrt{b/a}(1-(a-b))}{\sqrt{(a-b)^2 + a + b}}  , \pi\sqrt{\frac{(a-b)^2 + 2(a+b) +1}{(a-b)^2 + a + b}} \right).$$
\end{proof}

\begin{remark}
\begin{itemize}
\item We have 
$$ \{ m' (pv_2) + n' (ptv_1) \, : \, m',n' \in \zn\} \subset \Lambda_{\psi_{h,\rho}} ,$$
and the vectors $p v_2$ and $ptv_1$ are linearly independent.
\item If we consider $a=b=1/16$ and then $q^2=9$ (see Remark~\ref{rmk4}), we obtain that the
maps $\phi_{4/5,\rho}$ and $\phi_{4/5,0}$ determine two CMC proper-biharmonic
embeddings in $\sn^5$ with mean curvature $4/5$ from non-isometric tori. Similarly, choosing $a=b=1/36$ and then $q^2=(24)^2$, yield two CMC proper-biharmonic
embeddings in $\sn^5$ with mean curvature $144/145$ from non-isometric tori.
\end{itemize}
\end{remark}

From Proposition~\ref{prop?} and Proposition~\ref{prop??} we get

\begin{theorem}\label{thm7}
Let $h\in (0,1)$. Then there exists a CMC proper-biharmonic immersion from a torus $T^2$ into $\sn^5$, $\phi : T^2\to \sn^5$ with mean curvature $h$ if and only if either
\begin{itemize}
\item[i)] 
$$h = \frac{q^2 -1}{q^2 +1},$$
$q\in \qn$, $q>1$,\\
or
\item[ii)] 
$$ h = \frac{1- (a-b)^2}{1 + (a-b)^2 + 2(a+b)} ,$$
where $a=p^2/q^2$ and $b=r^2/t^2$ with $p,q,r,t \in \nn^*$, with the condition $0\leq (b-a)^2 <1$.
\end{itemize}
\end{theorem}

\begin{proof}
Let $\phi : T^2 \to \sn^5$ a CMC proper-biharmonic immersion with mean curvature $h\in (0,1)$. We consider the universal cover of $T^2$ and denote also with $\phi$ the CMC proper-biharmonic immersion $\phi = \phi\circ\pi : \rn^2 \to \sn^5$. After composing with isometries of $\rn^2$ and $\rn^6$, if necessary, the map $\psi = i\circ\psi : \rn^2 \to \rn^6$ coincides with a map $\psi_{h,\rho}$, with $\rho\in [0,\tfrac{\pi}{2})$. Now the result follows from 
Proposition~\ref{prop?} and Proposition~\ref{prop??}.
\end{proof}

\begin{example}
As an example one can take the values $m=1$, $m'=2$, $n=5$, $R_1=1$, $R'_1=R'_2=\frac{1}{2}$, $\mu_1 =1$ and letting $|H|=h\in (0,1)$ be free. Then
$$ \eta^2_1 +\eta^2_2 = \frac{2(h-1)}{h+1} .$$
As $s=\tfrac{1}{2}$, $\eta_1 = \frac{\sqrt{h}}{\sqrt{h+1}} + \i \frac{1}{\sqrt{h+1}}$ and $\eta_2 = \frac{\sqrt{h}}{\sqrt{h+1}} - \i \frac{1}{\sqrt{h+1}}$ then $\{\pm \eta_1 ,\pm \eta_2\}$ has 4 distinct elements. Then $\eta_1 z -\bar{\eta}_1 \bar{z} = 2\i(\frac{x}{\sqrt{h+1}} + \frac{\sqrt{h}}{\sqrt{h+1}}y)$, so
$$e^{\tfrac{\sqrt{\lambda_2}}{2}(\eta_1 z -\bar{\eta}_1 \bar{z})} = e^{\i(x+\sqrt{h}y)\sqrt{2}}$$ 
and
$$\frac{1}{\sqrt{2}} 2 \Re \Big(\sqrt{R'_1}e^{\tfrac{\sqrt{\lambda_2}}{2}(\eta_1 z -\bar{\eta}_1 \bar{z})} W_1\Big) =
\frac{1}{2}\Big( 0,0,\cos((x+\sqrt{h}y)\sqrt{2}),\sin((x+\sqrt{h}y)\sqrt{2}),0,0  \Big)$$
and similarly $\eta_2 z -\bar{\eta}_2 \bar{z} = 2\i(-\frac{x}{\sqrt{h+1}} + \frac{\sqrt{h}}{\sqrt{h+1}}y)$, 
$$\frac{1}{\sqrt{2}}\sqrt{R'_2}e^{\tfrac{\sqrt{\lambda_2}}{2}(\eta_2 z -\bar{\eta}_2 \bar{z})} = \frac{1}{2}e^{\i(-x+\sqrt{h}y)\sqrt{2}}$$ 
and
$$\frac{1}{\sqrt{2}} 2 \Re \Big(\sqrt{R'_2}e^{\tfrac{\sqrt{\lambda_2}}{2}(\eta_2 z -\bar{\eta}_2 \bar{z})} W_2\Big) =
\frac{1}{2}\Big( 0,0,0,0,\cos((-x+\sqrt{h}y)\sqrt{2}),\sin((-x+\sqrt{h}y)\sqrt{2}) \Big) .$$
From these we get
\begin{align*}
\psi &= \frac{1}{\sqrt{2}} (e^{\i\sqrt{2(1-h)}y},0,0) + \frac{1}{2}(0,e^{\i\sqrt{2}(x+\sqrt{h}y)},0) + \frac{1}{2}(0,0,e^{\i\sqrt{2}(-x+\sqrt{h}y)}) .
\end{align*}
From 
$$\langle \i \sqrt{\lambda_1} , v\rangle \equiv \langle \i \sqrt{\lambda_2}\bar{\eta_1} , v\rangle \equiv  \langle \i \sqrt{\lambda_2}\bar{\eta_2} , v\rangle \equiv 0 \pmod{2\pi}$$
we get
$$\Lambda_{\psi} = \Big\{ n v_1 + m v_2 \, : \,  m,n\in \zn \mbox{ s.t. } 2m \sqrt{\tfrac{h}{1-h}} \in \zn \Big\},$$
where 
$$v_1 = \sqrt{2} \pi (1,0),\quad \mbox{and} \quad v_2 = \sqrt{2}\pi \left(- \sqrt{\tfrac{h}{1-h}} , \tfrac{1}{\sqrt{1-h}} \right).$$
Clearly $\rank{\Lambda_{\psi}} =2$ if and only if $h=\tfrac{1}{4b+1}$, where $b=\tfrac{r^2}{t^2}$, $r,t\in \nn^*$.

If $b = \tfrac{1}{4}$, then $h=\tfrac{1}{2}$ and 
$$\Lambda_{\psi} = \Big\{ n v_1 + m v_2 \, : \,  m,n\in \zn \Big\},$$
where $v_2 = \sqrt{2}\pi (-1,\sqrt{2})$, and therefore
$$\Lambda_{\psi} = \Big\{ \sqrt{2}\pi (n,\sqrt{2}m) \,:\, m,n\in \zn \Big\}.$$
This is the CMC proper-biharmonic embedding found by Sasahara in~\cite{Sasahara}.
\end{example}

\subsection{CMC proper-biharmonic flat surfaces in $\sn^{n}$, $n\geq 7$}

\begin{proposition}
Let $h\in (0,1)$. Then there exists a CMC proper-biharmonic immersion $\phi : \rn^2 \to \sn^{2n+1}$, $n\geq 7$, $n$ odd, with mean curvature $h$.
\end{proposition}

\begin{proof}
Let $\phi \in (0,\tfrac{\pi}{2})$. From Lemma~\ref{lemma1} and Lemma~\ref{lemma2} we know that there exist $s$ and $\eta_2$ such that
$$ s \eta_1^2 + (1-s) \eta_2^2 = - \tfrac{1-h}{1+h} .$$
Then, we can see that 
\begin{align*}
& hs (\i \eta_1)^2 + h (1-s) (\i\eta_2)^2 + (1-h) \i^2 = \\
& -h \left( s \eta_1^2 + (1-s) \eta_2^2 \right) + (1-h) \i^2 = - \tfrac{1-h}{1+h} ,
\end{align*}
and $\{ \pm \i \eta_1 ,\pm \i \eta_2 ,\pm\i\}$ are six distinct complex numbers of norm one. Therefore, we have defined a CMC proper-biharmonic immersion from $\rn^2$ to $\sn^7$ with $|H|=h$.

Now, let $\phi : \rn^2 \to \sn^n$ be a CMC proper-biharmonic immersion with $|H|=h$ given by the parameters
$$ \sum_{j=1}^{m'} R'_{j} \eta_{j}^2 =  - \tfrac{1-h}{1+h} , \quad n=2m'+1.$$
By continuity, using again Lemma~\ref{lemma1} and Lemma~\ref{lemma2}, we can take $\eta_{m'+1}$, $\eta_{m'+2}$ and $s$ such that
$$ s \eta_{m'+1}^2 + (1-s) \eta_{m'+2}^2 = - \tfrac{1-h}{1+h} ,$$
and
$$\eta_{m'+1} \neq \pm\eta_j , \quad \eta_{m'+2} \neq \pm\eta_j , \quad \forall j =1,\dots, m'.$$
Then 
$$ (1/2)\sum_{j=1}^{m'} R'_j \eta_j^2 + \tfrac{s}{2} \eta_{m'+1}^2 + \tfrac{1-s}{2} \eta_{m'+2}^2 = - \tfrac{1-h}{1+h} ,$$
and $(1/2)\sum_{j=1}^{m'} R'_j + \tfrac{s}{2} + \tfrac{1-s}{2} =1$ while $\{\pm\eta_j,\pm\eta_{m'+1},\pm\eta_{m'+2}\}$ are $2(m'+2)$ distinct complex numbers of norm one. Therefore, we have defined a CMC proper-biharmonic immersion from $\rn^2$ to $\sn^{n+4}$ with $|H|=h$. Now, as we have constructed, for any $h\in (0,1)$, CMC proper-biharmonic immersions from $\rn^2$ to $\sn^{5}$ and $\sn^7$ with mean curvature equal to $h$, the proposition follows.
\end{proof}

In the following we will change the point of view and, given a torus, ask if there is a CMC proper-biharmonic immersion from that torus to $\sn^n$.

Let $T^2 = \rn^2 / \Lambda$, where $\Lambda$ is a discrete lattice of rank 2. We denote 
$$\Lambda^* = \{ w \in \rn^2 : \la w,z\ra \equiv 0 \pmod{2\pi} , \forall z\in \Lambda\}$$
and
$$c(\Lambda , r) = \{ w^2 : w\in \Lambda^* \mbox{ and } |w|=r\}.$$
Let $\phi_{h} : \rn^2 \to \sn^n$ be a CMC proper-biharmonic immersion, $h\in (0,1)$. This map quotients to $\phi_{h} : T^2 = \rn^2 / \Lambda \to \sn^n$ if and only if $\i \sqrt{\lambda_1} \bar{\mu}_k \in \Lambda^*$, $\forall k=1,\dots ,m$, and $\i \sqrt{\lambda_2} \bar{\eta}_j \in \Lambda^*$, $\forall j=1,\dots ,m'$. 
In this case $\Lambda$ is an abelian subgroup of $\Lambda_{\psi_h}$ and, if $\Lambda=\Lambda_{\psi_h}$, $\phi_h$ is an embedding from $T^2$.

If we denote $\alpha_k = (\pm \i \sqrt{\lambda_1} \bar{\mu}_k)^2= - \lambda_1 \bar{\mu}_k^2$ and $\gamma_j = (\pm \i \sqrt{\lambda_2} \bar{\eta}_j)^2 = - \lambda_2 \bar{\eta}_j^2$, we have that $\{ \alpha_k\}_{k=1}^m$ are $m$ distinct complex numbers, $|\alpha_k|=\lambda_1$ and $\{ \gamma_j\}_{j=1}^{m'}$ are $m'$ distinct complex numbers with 
$|\gamma_j| = \lambda_2$. Thus $\{ \alpha_k\}_{k=1}^m \subset c(\Lambda, \sqrt{\lambda_1})$ and $\{ \gamma_j\}_{j=1}^{m'}\subset c(\Lambda, \sqrt{\lambda_2})$. We can see that the relation (f) of Theorem~\ref{thmMiyata} can be re-written equivalently as
$$\sum_{k=1}^m \alpha_k R_k + \sum_{j=1}^{m'} \gamma_j R'_j =0.$$

Now we can state

\begin{theorem}
Let $\Lambda \subset \rn^2$ be a discrete lattice of rank 2 and $\Lambda$ its dual. Let $h\in (0,1)$ be a fixed constant. Then the flat torus $T=\rn^2/\Lambda$ admits a CMC proper-biharmonic immersion with mean curvature $h$ in $\sn^n$ if and only if there exist $m$ distinct complex numbers $\{ \alpha_k\}_{k=1}^m \subset c(\Lambda, \sqrt{\lambda_1})$ and $m'$ distinct complex numbers 
$\{ \gamma_j\}_{j=1}^{m'}\subset c(\Lambda, \sqrt{\lambda_2})$, $n= 2m + 2m' -1$, such that 
$$\sum_{k=1}^m \alpha_k R_k + \sum_{j=1}^{m'} \gamma_j R'_j =0,$$
where $\lambda_1 = 2(1-h)$, $\lambda_2 = 2(1+h)$, $\sum_{k=1}^{m} R_k =1$, $\sum_{j=1}^{m'} R'_j =1$, $R_k>0$, $R'_j >0$.
\end{theorem}

The following non-existence result follows easily.

\begin{proposition}
Let $\Lambda$ be a given lattice of rank 2 and $\Lambda^*$ its dual. Let $h\in (0,1)$ be a fixed constant. If $\Lambda^* \cap \sn^{1}(\sqrt{2(1-h)}) = \emptyset$ or 
$\Lambda^* \cap \sn^{1}(\sqrt{2(1+h)}) = \emptyset$, then there is no proper-biharmonic immersion $\phi : \rn^2 /\Lambda \to \sn^n$ with $|H|=h$, for any $n\geq 7$.
\end{proposition}

\begin{corollary}
Let $\Lambda$ be a given lattice of rank 2 and $\Lambda^*$ its dual. Assume that $\Lambda^* = v_1 \zn + v_2 \zn$, where $0<|v_1|\leq |v|$, $\forall v\in \Lambda^*\setminus\{(0,0)\}$. If $|v_1|\geq \sqrt{2}$ then there is no CMC proper-biharmonic immersion $\phi : \rn^2 /\Lambda \to \sn^n$ with $|H|\in (0,1)$, for any $n\geq 7$.
\end{corollary}

\begin{proposition}
Let $\Lambda$ be a given lattice of rank 2 and $\Lambda^*$ its dual. Let $h\in (0,1)$ be a fixed constant. Assume that $\Lambda^* \cap \sn^{1}(\sqrt{2(1-h)}) \neq \emptyset$ and $\Lambda^* \cap \sn^{1}(\sqrt{2(1+h)}) \neq \emptyset$. Denote by $H(\Lambda,r)$ the convex hull of $c(\Lambda,r)$. We have
\begin{itemize}
\item[i)] If $(0,0)\in H(\Lambda,\sqrt{2(1-h)})$ and $(0,0)\in H(\Lambda,\sqrt{2(1+h)})$ then there exists $\phi : \rn^2 / \Lambda \to \sn^n$ a proper-biharmonic immersion with $|H|=h$ and, moreover, it is pseudo-umbilical.
\item[ii)] If $(0,0)$ does not belong to the convex hull of $c(\Lambda,\sqrt{2(1-h)})\cup c(\Lambda,\sqrt{2(1+h)})$ then there is no proper-biharmonic immersion $\phi : \rn^2 / \Lambda \to \sn^n$ with $|H|=h$, for any $n\geq 7$.
\end{itemize}
\end{proposition}

\end{document}